\documentclass[11pt]{amsart}
\usepackage[T1]{fontenc}
\usepackage{amsmath,amssymb,amsfonts}
\usepackage[all,cmtip]{xy}
\usepackage[english]{babel}

\usepackage{xcolor}

\usepackage{tikz,tikz-cd}

\usepackage{hyperref}

\newcommand{\ol}[1]{\overline{#1}}

\newcommand{\Sch}{\mathbf{Sch}}

\newcommand{\ft}{\mathrm{ft}}
\newcommand{\lft}{\mathrm{lft}}
\newcommand{\sep}{\mathrm{sep}}
\newcommand{\aff}{\mathrm{aff}}

\newcommand{\Ad}{\mathbf{Ad}}
\newcommand{\cV}{\mathcal{V}}

\newcommand{\ad}{\mathrm{ad}}

\DeclareMathOperator{\Spf}{Spf}
\DeclareMathOperator{\Proj}{Proj}

\DeclareMathOperator{\Spec}{Spec}
\DeclareMathOperator*{\colim}{colim}

\DeclareMathOperator{\cofib}{cofib}
\DeclareMathOperator{\fib}{fib}
\DeclareMathOperator{\mycup}{\smallsmile}

\DeclareMathOperator{\Pro}{Pro}
\DeclareMathOperator{\Fun}{Fun}
\newcommand{\Spaces}{\mathbf{Spc}}

\DeclareMathOperator{\Map}{Map}

\newcommand{\an}{\mathrm{an}}
\newcommand{\cont}{\mathrm{cont}}
\newcommand{\op}{\mathrm{op}}

\newcommand{\Sh}{\mathbf{Sh}}
\newcommand{\PSh}{\mathbf{PSh}}
\DeclareMathOperator{\Spa}{Spa}
\newcommand{\bD}{\mathbf{D}}

\newcommand{\prolim}[1]{\underset{#1}{\operatorname{``}\lim \operatorname{''}  \hspace{.2ex}  }}  
\newcommand{\lprolim}[1]{\operatorname{``}\lim \operatorname{''} \hspace{-1.9ex} {}_{#1} \hspace{.5ex} } 

\newcommand{\Sp}{\mathbf{Sp}}

\newcommand{\cO}{\mathcal{O}}

\newcommand{\cU}{{\mathcal U}}
\newcommand{\A}{{\mathbb A}}

\newcommand{\bP}{{\mathbb P}}

\newcommand{\Ab}{\mathbf{Ab}}

\newcommand{\Kan}{K^{\an}}

\def\Kcont{K^{\cont}}

\def\suppi{\;\mathrm{on}\; (\pi)}

\newcommand{\on}{\,\mathrm{on}\,}

\theoremstyle{plain}
\newtheorem{thm}{Theorem}[section]

\newtheorem{lem}[thm]{Lemma}
\newtheorem{lemma}[thm]{Lemma}
\newtheorem{cor}[thm]{Corollary}
\newtheorem{prop}[thm]{Proposition}

\theoremstyle{definition}
\newtheorem{ex}[thm]{Example}
\newtheorem{defn}[thm]{Definition}

\newtheorem{claim}[thm]{Claim}

\title{\textit{K}-theory of non-archimedean rings II}
\author{Moritz Kerz}
\address{Fakult\"at f\"ur Mathematik, Universit\"at Regensburg, 93040 Regensburg, Germany}
\email{moritz.kerz@ur.de}

\author{Shuji Saito}
\address{Graduate School of Mathematical
Sciences,
University of Tokyo,
3-8-1 Komaba,
Tokyo,
Japan}
\email{sshuji@msb.biglobe.ne.jp}

\author{Georg Tamme}
\address{Institut f\"ur Mathematik, Fachbereich 08, Johannes Gutenberg-Universit\"at Mainz, 55099 Mainz, Germany}
\email{georg.tamme@uni-mainz.de}

\thanks{The authors are supported by the DFG through CRC 1085 \textit{Higher Invariants} (Universit\"at Regensburg). 
The second author thanks the Department of Mathematics at the University of
Regensburg
for hospitality while this work was done.
He is supported by JSPS KAKENHI Grant (15H03606). The third author is also supported by the DFG through TRR 326 (Project-ID 444845124)}

\date{\today}

\begin{document}

\begin{abstract}
We study fundamental properties of analytic $K$-theory of Tate rings such as homotopy invariance, Bass fundamental theorem, Milnor excision, and descent for admissible coverings. 
\end{abstract}

\maketitle

\setcounter{tocdepth}{2}
\tableofcontents

\section{Introduction}

This note is a sequel to our note \cite{KSTI}, and we refer to that paper for a general introduction and motivation.
In part I we were mainly concerned with the construction of what we call analytic $K$-theory of Tate rings and a comparison with continuous $K$-theory as proposed by Morrow \cite{morrow-hist}. The main result there was a weak equivalence between continuous and analytic $K$-theory for regular noetherian Tate rings admitting a ring of definition satisfying a certain weak resolution of singularities property. Negative continuous $K$-theory was studied by Dahlhausen in \cite{Dahlhausen}. 

The goal here is to study more closely the properties of analytic $K$-theory. It turns out that analytic $K$-theory has good formal properties, similar to Weibel's homotopy $K$-theory, without any regularity assumption. However, we have to restrict to Tate rings admitting a noetherian, finite dimensional ring of definition. In particular, for such Tate rings we establish the analytic Bass fundamental theorem (Corollary~\ref{cor:BassFT}), pro-homotopy invariance (Corollary~\ref{cor:pro-homotopy-invariance}), Milnor excision (Theorem~\ref{thm:excision}), and descent for rational coverings of the associated adic space (Theorem~\ref{thm:descent}). The latter allows us to globalize the definition of analytic $K$-theory to adic spaces satisfying a  finiteness condition. In addition, we use our results to prove that analytic $K$-theory of such Tate rings is in fact weakly equivalent to the ``$\A^{1}$-localization'' of Morrow's continuous $K$-theory \cite{morrow-hist} (see Corollary~\ref{cor:A1-loc}). We thank the referee for suggesting to do so.

Our main tool is a fibre sequence relating the continuous $K$-theory of a ring of definition, or more generally a Raynaud type model, with the analytic $K$-theory of a Tate ring (Theorems~\ref{thm.weakfibresequence}, \ref{thm:fund-fib-model}).

\subsection*{Notation and conventions}
By a   $\pi$-adic (or just adic) ring we mean a topological ring $A_{0}$ whose topology is the $\pi$-adic one for an element $\pi \in A_{0}$. We usually assume that $A_{0}$ is complete. A topological ring $A$ is called a Tate ring if it admits an open subring $A_{0} \subseteq A$ and an element $\pi \in A_{0}$ such that the subspace topology on $A_{0}$ is the $\pi$-adic one, and $\pi$ is invertible in $A$. Such an $A_{0}$ is called a ring of definition, and $\pi$ is called a pseudo-uniformizer. We have $A = A_{0}[\tfrac{1}{\pi}]$. If $A$ is Tate, then a subring $A_{0} \subseteq A$ is a ring of definition if and only if it is open and bounded. The basic reference is \cite[Sec.~1]{Huber}.

We denote by $\Spaces$ and $\Sp$ the $\infty$-categories of spaces and spectra, respectively. $\Pro(\Sp)$ denotes the $\infty$-category of pro-spectra, see e.g.~\cite[Sec.~2]{KSTI}. A map of pro-spectra $f$ is a weak equivalence, if its level-wise truncation $\tau_{\leq n}f$ is an equivalence of pro-spectra for every integer $n$. See \cite[Sec.~2.2]{KSTI} for a discussion and equivalent characterizations.  Similarly, a sequence $X \to Y \to Z$ of pro-spectra is a weak fibre sequence if it is weakly equivalent to a fibre sequence.

The letter $K$ denotes the nonconnective algebraic $K$-theory functor for rings and quasi-compact, quasi-separated schemes. Its connective covering is denoted by $k$. For the following notations, we refer to  the corresponding definitions in \cite{KSTI}:
\begin{itemize}
\item [-] $K^{\cont}$ for an adic ring: Definition~5.1.
\item [-] $K^{\cont}$ for a Tate ring: Definition~5.3.
\item [-] $k^{\an}$ for an adic ring: Defintion~6.5.
\item [-] $k^{\an}$ for a Tate ring: Definition~6.11. (Warning: in general, this is not the connective cover of $K^{\an}$.)
\item [-] $K^{\an}$ for a Tate ring: Definition~6.15. 
\end{itemize}

\section{Recollections from part I, and complements}

\subsection{Pro-homotopy algebraic $K$-theory}
\label{sec:pro-homotopy-algebraic-K}

In this subsection we introduce a variant of Weibel's homotopy algebraic $K$-theory, which will be an important tool to compare the analytic $K$-theory of a Tate ring with the continuous $K$-theory of a ring of definition without assuming any regularity condition. See \cite[Sec.~6.1]{KSTI} for similar constructions in the setting of complete Tate rings.

Let $R$ be a ring, and let $\pi \in R$ be any element. For non-negative integers $p$, $j$  we define
\[
R[\Delta^{p}_{\pi^{j}} ] := R[t_{0}, \dots, t_{p}]/( t_{0} + \dots + t_{p} - \pi^{j}).
\]
For fixed $j$ and varying $p$, the $R[\Delta^{p}_{\pi^{j}}]$ form a simplicial ring, which we denote by $R[\Delta_{\pi^{j}}]$. Multiplication of the coordinates $t_{i}$ by $\pi$ induces maps of simplicial rings
$R[\Delta_{\pi^{j}}] \to R[\Delta_{\pi^{j-1}}]$, and in this way we obtain the pro-simplicial ring
\[
\prolim{j} R[\Delta_{\pi^{j}}].
\]

Let $\Sch_{R}$ denote the category of quasi-compact, quasi-separated (qcqs) $R$-schemes. Let $F$ be a functor $F \colon \Sch_{R}^{op} \to \Sp$. We define a new functor 
$F^{\pi^{\infty}} \colon \Sch_{R}^{op} \to \Pro(\Sp)$ as follows.
We let
$F^{\pi^{j}}(X)$ be
the geometric realization of the simplicial spectrum $[p] \mapsto F(X \otimes_{R} R[\Delta^{p}_{\pi^{j}}])$ and define
\[
F^{\pi^{\infty}}(X) = \prolim{j} F^{\pi^{j}}(X) \in \Pro(\Sp).
\]

\begin{defn}
Let $X$ be a qcqs $R$-scheme. We define its \emph{pro-homotopy algebraic $K$-theory (with respect to $\pi \in R$)} as the pro-spectrum
\[
K^{\pi^{\infty}}(X).
\]
Similarly we define a version with support on $X/(\pi) = X \otimes_{R} R/(\pi)$:
\[
K^{\pi^{\infty}} (X \on (\pi) )
\]
\end{defn}

For an $R$-scheme $X$ we write $X[t] = X \otimes_{R} R[t]$.
By construction, $K^{\pi^{\infty}}$ satisfies the following form of pro-homotopy invariance:
\begin{lemma}
	\label{lem:pro-HI-K-pi-infty}
For $X \in \Sch_{R}$ the canonical map 
\[
K^{\pi^{\infty}}(X) \to \prolim{t \mapsto \pi t} K^{\pi^{\infty}}(X[t])
\]
is an equivalence of pro-spectra. Similarly for $K^{\pi^{\infty}}(- \on (\pi))$.
\end{lemma}
The \textit{proof} is the same as that of \cite[Prop.~6.3]{KSTI}.
We will use the following easy observation:
\begin{lemma}	
	\label{lem:K-pi-infty-KH}
If $\pi$ is invertible on $X$, there is a natural equivalence 
\[
K^{\pi^{\infty}}(X) \cong KH(X),
\]
where $KH(X)$ denotes Weibel's homotopy $K$-theory spectrum. In particular, for any $X$ there is a fibre sequence of pro-spectra
\[
K^{\pi^{\infty}}(X \on (\pi) ) \to K^{\pi^{\infty}}(X) \to KH(X[\tfrac{1}{\pi}]).
\]
\end{lemma}
\begin{proof}
If $\pi$ is invertible on $X$, then the transition map $R[\Delta^{p}_{\pi^{j}}] \to R[\Delta^{p}_{\pi^{0}}] = R[\Delta^{p}]$ induces an isomorphism
$X\otimes_{R} R[\Delta^{p}] \to X \otimes_{R} R[\Delta^{p}_{\pi^{j}}]$. Thus the
pro-system defining $K^{\pi^{\infty}}(X)$ is constant with value $KH(X)$. The second claim
follows from the first one together with the localization sequence associated to the
 identity map $X \to X[\frac{1}{\pi}]$.
\end{proof}

\subsection{The fundamental fibre sequence}
	\label{sec:fund-fib}

The following weak fibre sequence and its generalization to non-affine models in
Theorem~\ref{thm:fund-fib-model} are our fundamental tool to control the analytic
$K$-theory of Tate rings.   For a $\pi$-adic ring $A_0$ we denote by
  $K^{\cont}(A_0)$ the pro-spectrum  $ \prolim{n} K(A_0/(\pi^n))$.  

\begin{thm}
	\label{thm.weakfibresequence}
Let $A_{0}$ be a complete $\pi$-adic ring, and let $A=A_{0}[1/\pi]$ be the associated Tate ring. Assume that $A_{0}$ is noetherian and of finite Krull dimension. Then there is a weak fibre sequence of pro-spectra
\[
K^{\pi^{\infty}}(A_{0} \on (\pi)) \to K^{\cont}(A_{0}) \to K^{\an}(A).
\]
\end{thm}

Before entering the proof, we recall the analytic Bass construction,  referring to \cite[Sec.~4.4]{KSTI} for details.
For a functor $E$ defined on the category of adic or Tate rings and with values in a cocomplete stable $\infty$-category (such as $\Sp$ or $\Pro(\Sp)$) we define the functor $\Lambda E$ by
\[
\Lambda E(A) = \fib( E(A\langle t \rangle) \sqcup_{E(A)} E (A\langle t^{-1} \rangle) \to E(A\langle t,t^{-1}\rangle) )
\]
If $E$ has the structure of a module over connective algebraic $K$-theory, then there is a natural map $\lambda\colon E(A) \to \Lambda E(A)$ given as the composite
\[
\lambda\colon E(A) \xrightarrow{\mycup t} \Omega E(A \langle t,t^{-1}\rangle) \to \Lambda E(A).
\]
If $\lambda$ is an equivalence, then $E$ satisfies an analytic version of the Bass fundamental theorem, see \cite[Prop.~4.13]{KSTI}.
The analytic Bass construction $E^{B}$ of $E$ is then given by
\[
E^{B}(A) = \colim ( E(A) \xrightarrow{\lambda}\Lambda E(A) \xrightarrow{\Lambda(\lambda)} \Lambda^{2}E(A) \to \dots).
\]

\begin{ex}
	\label{ex:k-cont-Bass}
Let $k^{\cont}$ be the functor from adic rings to pro-spectra that assigns to $A_{0}$ the
connective continuous $K$-theory \[ k^{\cont}(A_{0}) = \prolim{n} k(A_{0}/(\pi^{n})).\] As
a rule, we write $k$ for the connective algebraic $K$-theory $\tau_{\geq 0}K$. Then $k^{\cont, B}$ is canonically equivalent to non-connective continuous $K$-theory $K^{\cont}$. 
Indeed, using the isomorphisms $A_{0}\langle t \rangle / (\pi^{n}) \cong A_{0}/(\pi^{n})[t]$ and similarly for $t$ and $t^{-1}$, the classical Bass fundamental theorem \cite[Thm.~6.6]{thomason} implies that $\lambda\colon K^{\cont} \to \Lambda K^{\cont}$ is an equivalence, hence $K^{\cont} \simeq K^{\cont, B}$. Moreover, by nilinvariance of negative $K$-groups, the fibre of the canonical map $k^{\cont} \to K^{\cont}$ has values in constant, coconnective pro-spectra. Hence this fibre vanishes upon applying the analytic Bass construction \cite[Lem.~4.15]{KSTI}.
\end{ex}

\begin{proof}[Proof of Theorem~\ref{thm.weakfibresequence}]
As in \cite[Sec.~3.1]{KSTI} we denote by  $A_{0}\langle \Delta^{p}_{\pi^{j}}\rangle$  the $\pi$-adic completion of the ring $A_{0}[\Delta^{p}_{\pi^{j}}]$.
In \cite[Def.~6.5]{KSTI} we defined the pro-spectrum
\[
k^{\an}(A_{0}) = \prolim{j} k(A_{0}\langle \Delta_{\pi^{j}}\rangle)
\]
where $ k(A_{0}\langle \Delta_{\pi^{j}}\rangle)$ is the geometric realization of the simplicial spectrum $[p] \mapsto  k(A_{0}\langle \Delta^{p}_{\pi^{j}}\rangle)$. Similarly, we set\footnote{in \cite{KSTI}, this pro-spectrum is denoted $\bar k^{\an}(A_{0})$.}
\[
k^{\an}(A_{0} \on (\pi) ) = \prolim{j} k(A_{0}\langle \Delta_{\pi^{j}}\rangle \on (\pi) ),
\]
where $k(- \on (\pi))$ is the connective covering of $K(-\on (\pi))$.
There is a natural map
$k^{\an}(A_{0} \on (\pi) ) \to k^{\an}(A_{0} )$
and the pro-spectrum $\tilde k^{\an}(A;A_{0})$ is defined to fit in the cofibre sequence
\[
k^{\an}(A_{0} \on (\pi) ) \to k^{\an}(A_{0} ) \to \tilde k^{\an}(A;A_{0}).
\] 
By \cite[Thm.~6.9]{KSTI}, we have a natural weak equivalence $k^{\an}(A_{0}) \simeq k^{\cont}(A_{0})$, and thus we have a natural weak fibre sequence
\begin{equation*}
	\label{seq:connective}
k^{\an}(A_{0} \on (\pi) ) \to k^{\cont}(A_{0} ) \to \tilde k^{\an}(A;A_{0}).
\end{equation*}
We claim that applying the analytic Bass construction $(-)^{B}$ yields the asserted weak fibre sequence. 
Note that $(-)^{B}$ preserves weak equivalences and weak fibre sequences of functors as its construction only involves colimits.

Firstly, by \cite[Lem.~6.18(i)]{KSTI} there is a natural weak equivalence $(\tilde k^{\an})^{B}(A;A_{0}) \simeq K^{\an}(A)$. Secondly, by Example~\ref{ex:k-cont-Bass} there is an equivalence $k^{\cont, B}(A_{0}) \cong K^{\cont}(A_{0})$.
It thus remains to prove that 
the analytic Bass construction of $k^{\an}(- \on (\pi))$ evaluated on $A_{0}$ is weakly equivalent to $K^{\pi^{\infty}}(A_{0} \on (\pi))$.

Note that by Thomason's excision \cite[Prop.~3.19]{thomason} we have equivalences 
\begin{align*}
K(A_{0}\langle \Delta^{p}_{\pi^{j}} \rangle \on (\pi) ) &\simeq K(A_{0}[ \Delta^{p}_{\pi^{j}} ] \on (\pi) ) ,\\
K(A_{0}\langle \Delta^{p}_{\pi^{j}} \rangle\langle t \rangle \on (\pi) ) &\simeq K(A_{0}[ \Delta^{p}_{\pi^{j}} ][t] \on (\pi) )
\end{align*}
and similarly for $t^{-1}$ and $t,t^{-1}$. By Weibel's $K$-dimension conjecture, proven in \cite{KST-Weibel}, all these spectra are $(-d-1)$-connective for $d=\dim(A_{0})$. The classical Bass fundamental theorem and excision thus imply that 
\[
\Lambda^{n}k(A_{0}\langle \Delta^{p}_{\pi^{j}} \rangle \on (\pi) ) \simeq K(A_{0}\langle \Delta^{p}_{\pi^{j}} \rangle \on (\pi) ) 
\simeq K(A_{0}[ \Delta^{p}_{\pi^{j}} ] \on (\pi) ) 
\]
for all $p$ and $j$ and every $n > d$.   As finite limits and finite colimits in $\Pro(\Sp)$
  are computed level-wise and as  geometric realizations in $\Sp$ commute with finite limits and colimits this shows that
\[
    \Lambda^n k^{\an}(A_0 \on (\pi) ) \simeq
    \Lambda^n K^{\pi^{\infty}}( A_{0} \on (\pi)) \simeq K^{\pi^{\infty}}( A_{0} \on (\pi))
\]
for  $n > d$.
So the diagram in the definition of the analytic Bass construction for $E=k^{\an}(- \on
(\pi) )$, evaluated on $A_{0}$, is eventually constant with value $K^{\pi^{\infty}}( A_{0}
\on (\pi))$. This concludes the proof.
\end{proof}

We record the following applications:
\begin{cor}[{Bass fundamental theorem for analytic $K$-theory}]    
	\label{cor:BassFT}
	Let $A$ be a complete Tate ring which admits a finite dimensional, noetherian ring of definition. Then the map 
	\(
	\lambda\colon K^{\an}(A) \to \Lambda K^{\an}(A)
	\)
	is a weak equivalence of pro-spectra, and hence the Bass fundamental theorem holds. That is, for every integer $i$, we have an exact sequence
	\[
	0 \to K_{i}^{\an}(A) \to K_{i}^{\an}(A\langle t \rangle) \oplus K_{i}^{\an}(A\langle t^{-1} \rangle)
		\to K_{i}^{\an}(A\langle t, t^{-1} \rangle) \to K_{i-1}^{\an}(A) \to 0
	\]
	and the right-hand map has a splitting given by multiplication with the class of $t \in K_{1}(A\langle t, t^{-1} \rangle)$.
\end{cor}
\begin{proof}
Let $A_{0}$ be a finite dimensional noetherian ring of definition. 
As in Example~\ref{ex:k-cont-Bass}, the map $\lambda\colon K^{\cont}(A_{0}) \to \Lambda K^{\cont}(A_{0})$ is an equivalence. 
Using again Thomason's excision and the classical Bass fundamental theorem we see that also $\lambda\colon K^{\pi^{\infty}}(A_{0} \on (\pi) ) \to \Lambda K^{\pi^{\infty}}(A_{0} \on (\pi) )$ is an equivalence. By Theorem~\ref{thm.weakfibresequence} this implies that $\lambda\colon K^{\an}(A) \to \Lambda K^{\an}(A)$ is a weak equivalence. The rest now follows from \cite[Prop.~4.13]{KSTI}.
\end{proof}

\begin{cor}[Pro-homotopy invariance]
	\label{cor:pro-homotopy-invariance}
	Let $A$ be a complete Tate ring which admits a finite dimensional, noetherian ring of definition. Then the map 
	\[
	K^{\an}(A) \to \prolim{t\mapsto \pi t} K^{\an}(A\langle t \rangle)
	\]
	is a weak equivalence.
\end{cor}
\begin{proof}
Both other terms in the fundamental fibre sequence of Theorem~\ref{thm.weakfibresequence} satisfy this type of pro-homotopy invariance: See \cite[Lem.~5.13]{KSTI} for $K^{\cont}$ and Lemmas~\ref{lem:pro-HI-K-pi-infty}, \ref{lem:K-pi-infty-KH} for $K^{\pi^{\infty}}(- \on (\pi))$.
\end{proof}

We want to compare analytic $K$-theory of Tate rings to the ``$\A^{1}$-locali\-zation'' of Morrow's continuous $K$-theory \cite{morrow-hist} (see \cite[Def.~5.3]{KSTI}), similarly as Weibel's homotopy $K$-theory is the $\A^{1}$-localization of algebraic $K$-theory. Recall that the $\infty$-category $\Pro(\Sp)$ is cocomplete, so it admits geometric realizations of simplicial objects. Note also that general colimits in $\Pro(\Sp)$ cannot be computed level-wise. However, up to weak equivalence this is the case for geometric realizations of bounded below simplicial objects, as we prove in the following lemma.
Let $X_{\bullet} = \lprolim{\lambda} X_{\bullet}^{(\lambda)}$ be a pro-object of simplicial spectra, i.e.~an object of $\Pro(\Fun(\Delta^{\op}, \Sp))$.
We can view $X_{\bullet}$ as a simplicial object in $\Pro(\Sp)$ and form its geometric realization $|X_{\bullet}|$ there, or we can form the level-wise  geometric realizations $|X_{\bullet}^{(\lambda)}|$  in spectra and consider the pro-spectrum $\lprolim{\lambda} |X_{\bullet}^{(\lambda)}|$.

\begin{lem}\label{lem:geom-realization-pro-spectra}
In the above situation, assume that $X_{\bullet}$ is uniformly bounded below, i.e.~that there exists an integer $N$ such that $X_{p}^{(\lambda)} \in \Sp_{\geq N}$ for all $p\in \Delta$ and all $\lambda \in \Lambda$. Then the evident map
\[
|X_{\bullet}| \to \prolim{\lambda} |X_{\bullet}^{(\lambda)}|
\]
is a weak equivalence.
\end{lem}
\begin{proof}
It is easy to check that the $\infty$-category $\Pro(\Sp)$ admits a $t$-structure whose connective part $\Pro(\Sp)_{\geq 0}$ is given by the essential image of the evident functor $\Pro(\Sp_{\geq 0}) \to \Pro(\Sp)$ and whose coconnective part $\Pro(\Sp)_{\leq 0}$ is the essential image of the functor $\Pro(\Sp_{\leq 0}) \to \Pro(\Sp)$. Its heart is the category of pro-abelian groups $\Pro(\Ab)$. The homotopy group functors of the $t$-structure $\pi_{n} \colon \Pro(\Sp) \to \Pro(\Ab)$ are just the level-wise homotopy group functors $\Pro(\pi_{n})$, which we previously also denoted by $\pi_{n}$.

By shifting $X_{\bullet}$, we may assume that $N=0$ so that $X_{\bullet}$ determines a simplicial object of $\Pro(\Sp)_{\geq 0}$. In this situation  there is a convergent first quadrant spectral sequence of pro-abelian groups $E^{1}_{p,q} \Rightarrow \pi_{p+q}(|X_{\bullet}|)$ where $E^{1}_{*,q}$ is the normalized chain complex associated with the simplicial pro-abelian group $\pi_{q}(X_{\bullet})$, see \cite[Prop.~1.2.4.5]{HA}. On the other hand, for each $\lambda$ we have a convergent spectral sequence of abelian groups $E^{1,(\lambda)}_{p,q} \Rightarrow \pi_{p+q}(|X^{(\lambda)}_{\bullet}|)$ where again the $E^{1}$-terms are given by the normalized chain complexes associated with the simplicial abelian group $\pi_{q}(X^{(\lambda)}_{\bullet})$. As finite limits and finite colimits in $\Pro(\Ab)$ are computed level-wise, the pro-system of these spectral sequences yields a convergent spectral sequence $\lprolim{\lambda} E^{1,(\lambda)}_{p,q} \Rightarrow \lprolim{\lambda} \pi_{p+q}(|X_{\bullet}^{(\lambda)}|)$. By the same reason, and as $\pi_{q}(X_{p}) = \lprolim{\lambda} \pi_{q}(X_{p}^{(\lambda)})$, the canonical map of spectral sequences $E^{*}_{*,*} \to \lprolim{\lambda} E^{*,(\lambda)}_{*,*}$ is an isomorphism on $E^{1}$-pages. Hence $\pi_{n}(|X_{\bullet}|) \to \lprolim{\lambda} \pi_{n}(|X_{\bullet}^{(\lambda)}|)$ is an isomorphism of pro-abelian groups, as was to be shown.
\end{proof}

Now let $A$ be a Tate ring with ring of definition $A_{0}$ and pseudo-uniformizer $\pi$. Recall that the continuous $K$-theory of $A$ with respect to the ring of definition $A_{0}$ is given by
\[
K^{\cont}(A; A_{0}) = \cofib(K(A_{0} \text{ on } (\pi)) \to K^{\cont}(A_{0})),
\]
and this pro-spectrum is independent of the choice of $A_{0}$ up to weak equivalence \cite[Def.~5.3, Prop.~5.4]{KSTI}. As this definition is clearly functorial in the pair $(A_{0},\pi)$, we get a well defined simplicial pro-spectrum $K^{\cont}(A\langle \Delta_{\pi^{j}}^{\bullet} \rangle; A_{0}\langle \Delta_{\pi^{j}}^{\bullet} \rangle)$ (which can also be realized as a pro-simplicial spectrum) and we define $K^{\cont}(A\langle \Delta_{\pi^{j}} \rangle; A_{0}\langle \Delta_{\pi^{j}} \rangle)$ to be its geometric realization. As weak equivalences are preserved under colimits, this definition does not depend on the choice of $A_{0}$ up to weak equivalence. Taking the colimit over all rings of definition $A_{0}$ gives the well defined pro-spectrum $K^{\cont}(A\langle \Delta_{\pi^{j}} \rangle)$. The ``$\A^{1}$-localization'' of continuous $K$-theory at the Tate ring $A$ is then the pro-spectrum 
\[
\prolim{j} K^{\cont}(A\langle \Delta_{\pi^{j}} \rangle).
\]

\begin{cor} [{$\A^{1}$-localization}]
\label{cor:A1-loc}
Suppose the Tate ring $A$ admits a noetherian, finite dimensional ring of definition. Then there is a canonical weak equivalence of pro-spectra
\[
\prolim{j} K^{\cont}(A\langle \Delta_{\pi^{j}} \rangle) \simeq K^{\an}(A).
\]
\end{cor}
\begin{proof}
Let $A_{0} \subseteq A$ be a noetherian, finite dimensional ring of definition, and let $\pi \in A_{0}$ be a pseudo-uniformizer.
By the discussion above, we have a weak fibre 
sequence of pro-simplicial spectra
\begin{equation}\label{eq:eee}
K(A_{0}\langle\Delta_{\pi^{j}}^{\bullet}\rangle \text{ on } (\pi))\to 
K^{\cont}(A_{0}\langle\Delta_{\pi^{j}}^{\bullet}\rangle) \to
K^{\cont}(A\langle\Delta_{\pi^{j}}^{\bullet}\rangle; A_{0}\langle\Delta_{\pi^{j}}^{\bullet}\rangle)
\end{equation}
As in the proof of Theorem~\ref{thm.weakfibresequence}, the the first two and hence the third pro-simplicial spectrum in this fibre sequence are uniformly bounded below by $-\dim(A_{0})$.
Passing to geometric realizations and then to the limit $\lprolim{j}$, the first term gets identified with $K^{\pi^{\infty}}(A_{0} \text{ on } (\pi))$. For the second term we compute
\begin{align*}
\prolim{j} |K^{\cont}(A_{0}\langle \Delta^{\bullet}_{\pi^{j}} \rangle)| &= \prolim{j} | \prolim{n} K(A_{0}\langle \Delta^{\bullet}_{\pi^{j}} \rangle/(\pi^{n}) | \\
&\simeq \prolim{j}\prolim{n} |K(A_{0}\langle \Delta^{\bullet}_{\pi^{j}} \rangle/(\pi^{n}) )| &&\text{(Lemma~\ref{lem:geom-realization-pro-spectra})} \\
&\simeq \prolim{n}\prolim{j} |K(A_{0}\langle \Delta^{\bullet}_{\pi^{j}} \rangle/(\pi^{n})) |   \\
&\simeq \prolim{n} K(A_{0}/(\pi^{n})) &&\text{(*)}\\
&= K^{\cont}(A_{0}).
\end{align*}
Here the equivalence (*) follows from the isomorphism of pro-simplicial rings $\lprolim{j} A_{0}\langle \Delta^{\bullet}_{\pi^{j}} \rangle/(\pi^{n}) \cong A_{0}/(\pi^{n})$, where we view the right-hand term as a constant pro-simplicial ring. Thus the weak fibre sequence \eqref{eq:eee} gives rise to the weak fibre sequence
\[
K^{\pi^{\infty}}(A_{0} \text{ on } (\pi)) \to K^{\cont}(A_{0}) \to \prolim{j} K^{\cont}(A\langle \Delta_{\pi^{j}}\rangle).
\]
Comparing with the weak fibre sequence of Theorem~\ref{thm.weakfibresequence}, we obtain the requested weak equivalence.
\end{proof}


There is a more general version of the fundamental fibre sequence involving a Raynaud model $X$ of the Tate ring $A =A_{0}[\frac{1}{\pi}]$. We will use this to prove descent for analytic $K$-theory in Section~\ref{sec:descent}. Let $p\colon X \to \Spec(A_{0})$ be an admissible morphism, i.e.~a proper morphism of schemes which is an isomorphism over $\Spec(A)$. Let
\[
K^{\cont}(X) = \prolim{j} K(X \otimes_{A_{0}} A_{0}/(\pi^{j}) ).
\]

\begin{thm}
	\label{thm:fund-fib-model}
Let $A_{0}$ be a complete $\pi$-adic ring which is noetherian and of finite Krull dimension, and let $A = A_{0}[\frac{1}{\pi}]$. Let $p\colon X \to \Spec(A_{0})$ be an admissible morphism. Then there is a weak fibre sequence of pro-spectra
\[
K^{\pi^{\infty}} (X \on (\pi)) \to K^{\cont}(X) \to K^{\an}(A).
\]
\end{thm}
\begin{proof}
Thanks to Theorem \ref{thm.weakfibresequence}, it suffices to show that in $\Pro(\Sp)$ there is a weakly cartesian square
\begin{equation}\label{eq;thm.KHancont-variant}\begin{split}
\xymatrix{
K^{\pi^{\infty}}(A_{0} \on (\pi)) \ar[r]\ar[d] & K^{\cont}(A_{0}) \ar[d]\\
K^{\pi^{\infty}}(X \on (\pi)) \ar[r] & K^{\cont}(X) .
}
\end{split}
\end{equation}
We have a commutative diagram 
\begin{equation}\label{diag:aaaa}\begin{split}
\xymatrix{
K^{\pi^{\infty}}(A_{0} \on (\pi)) \ar[r]\ar[d] & K^{\pi^{\infty}}(A_{0} ) \ar[r]\ar[d] &  KH(A)\ar[d]^{\cong}\\
K^{\pi^{\infty}}(X \on (\pi)) \ar[r] & K^{\pi^{\infty}}(X) \ar[r] &  KH(X [\tfrac{1}{\pi}])
}
\end{split}
\end{equation}
where the rows are fibre sequences in $\Pro(\Sp)$ provided by Lemma~\ref{lem:K-pi-infty-KH}. 
Since the right vertical map is an equivalence,  the left-hand square is cartesian. 
Write $X/(\pi^m)=X\otimes_{A_{0}} A_{0}/(\pi^m)$. We claim that the square
\begin{equation}\label{diag:KH-KBcont-fibre-square}\begin{split}
\xymatrix{
K^{\pi^{\infty}}(A_{0}) \ar[r]\ar[d] &  \prolim{j,m} K(A_{0}/(\pi^{m})[\Delta_{\pi^{j}}]) \ar[d]\\
K^{\pi^{\infty}}(X) \ar[r] &  \prolim{j,m} K(X/(\pi^{m})[\Delta_{\pi^{j}}]) \\
}
\end{split}
\end{equation}
in $\Pro(\Sp)$ is weakly cartesian. For this let $Q(p,j,m)$ be the square of spectra
\[
\xymatrix{
K(A_{0}[\Delta^{p}_{\pi^{j}}]) \ar[r]\ar[d] &   K(A_{0}/(\pi^{m})[\Delta^{p}_{\pi^{j}}]) \ar[d]\\
K(X[\Delta^{p}_{\pi^{j}}]) \ar[r] &  K(X/(\pi^{m})[\Delta^{p}_{\pi^{j}}]) \\
}
\]
and denote by $F(p,j,m)$ its total fibre. 
So the square \eqref{diag:KH-KBcont-fibre-square} is 
\[
\prolim{j,m} \colim_{p\in\Delta^{\op}} Q(p,j,m).
\]
By ``pro cdh descent'' \cite[Thm.~A]{KST-Weibel}, the square of pro-spectra $\lprolim{m} Q(p,j,m)$ is weakly cartesian for every $p$ and every $j$, i.e.~$\lprolim{m} F(p,j,m) \simeq 0$. Note also that by Weibel's $K$-dimension conjecture \cite[Thm.~B]{KST-Weibel} all spectra in $Q(p,j,m)$ are $(-d)$-connective for $d=\dim(A_{0})$, hence $F(p,j,m)$ is $(-d-2)$-connective. The spectral sequence computing the homotopy groups of the geometric realization of a simplicial spectrum now implies that the pro-spectrum
\[
\prolim{m} \colim_{p\in \Delta^{op}} F(p,j,m)
\]
is weakly contractible for every $j$. Taking $\lprolim{j}$ implies that \eqref{diag:KH-KBcont-fibre-square} is weakly cartesian.

Finally, consider the commutative square in $\Pro(\Sp)$:
\begin{equation}\label{diag:bbbb} \begin{split}
\xymatrix{
K(A_{0}/(\pi^m))\ar[d]\ar[r]^-{\simeq} & \prolim{j} K(A_{0}/(\pi^m)[\Delta_{\pi^j}]) \ar[d] \\
K(X/(\pi^m)) \ar[r]^-{\simeq} &\prolim{j} K(X/(\pi^m)[\Delta_{\pi^j}])\;. 
}
\end{split}
\end{equation}
The horizontal maps in the above diagram are equivalences by the fact that 
the map $A_{0}/(\pi^{m}) \to \lprolim{j} A_{0} / (\pi^{m})[\Delta_{\pi^j}]$ is an isomorphism of pro-simplicial rings and a similar fact with $A_{0}$ replaced by $X$.

Composing the left-hand square in \eqref{diag:aaaa}, \eqref{diag:KH-KBcont-fibre-square}, and the horizontal inverse of \eqref{diag:bbbb}, we obtain the desired weakly cartesian square \eqref{eq;thm.KHancont-variant}.
\end{proof}

\section{Excision for analytic $K$-theory}

The goal of this section is to establish excision for Milnor squares for analytic $K$-theory for Tate rings that admit  a finite dimensional, noetherian ring of definition.

Recall that a \emph{Milnor square} is a cartesian square of rings of the form
\begin{equation} \label{Milnor-square}
	\begin{tikzcd}
	A \ar[r] \ar[d,"\phi"] & \ol A \ar[d] \\
	B \ar[r] & \ol B
	\end{tikzcd}
\end{equation}
in which the horizontal maps are surjective.
By a \emph{Milnor square of complete Tate rings} we mean a Milnor square in which all rings are complete Tate rings and all homomorphisms are continuous.

\begin{lemma}
	\label{lem:milnor-square-ring-of-def}
Assume that \eqref{Milnor-square} is a Milnor square of complete Tate rings. Then there exist rings of definition $A_{0}$, $B_{0}$, $\ol A_{0}$, and $\ol B_{0}$ such that \eqref{Milnor-square} restricts to a Milnor square
\[
\begin{tikzcd}
A _{0}\ar[r] \ar[d,"\phi"] & \ol A_{0} \ar[d] \\
B_{0} \ar[r] & \ol B_{0}.
\end{tikzcd}
\]
If $A$ and $B$ admit finite dimensional noetherian rings of definition, one can choose these rings of definition to be noetherian and finite dimensional as well.
\end{lemma}

\begin{proof}
First note that if $A$ admits the noetherian ring of definition $A_{0}$ and $A_{1}$ is another ring of definition containing $A_{0}$, then $A_{1}$ is finite over $A_{0}$ and in particular noetherian with $\dim(A_{0}) = \dim(A_{1})$. Indeed, if $\pi \in A_{0}$ is a pseudo-uniformizer, then by boundedness $A_{1}$ is contained in the free $A_{0}$-module $\pi^{-n}A_{0}$ for $n$ big enough.
Keeping this in mind, the second claim will be clear from the following construction.

Choose rings of definition $A_{0}$ and $B_{0}$ of $A$ and $B$, respectively. Replacing
$B_{0}$ by $B_{0}\cdot \phi(A_{0})$ if necessary,  we may assume that $\phi$ restricts to
a map $A_{0} \to B_{0}$. Now let $\ol A_{0}$ be the image of $A_{0}$ in $\ol A$ and   let
$\ol B_{0}$ be the image of $B_0$ in $\ol B$. Define $A_{1} \subseteq A$ as the pullback
\[
A_{1} = B_{0} \times_{\ol B_{0}}\ol A_{0}.
\]
By Banach's open mapping theorem (see \cite{Henkel:2014to} for a version for Tate rings) the continuous isomorphism $A \to B \times_{\ol B} \ol A$ is a homeomorphism. As $B_{0} \subseteq B$ and $\ol A_{0} \subseteq \ol A$ are bounded, this implies that $A_{1} \subseteq A$ is bounded. As it also contains $A_{0}$, it is open and hence a ring of definition. Replacing $A_{0}$ by $A_{1}$ we have constructed the desired rings of definition.
\end{proof}

\begin{thm}
	\label{thm:excision}
Let \eqref{Milnor-square} be a Milnor square of complete Tate rings. Assume that $A$ and $B$ admit noetherian rings of definition that are finite dimensional. Then the induced square of pro-spectra
\[
  \tag{$\boxed{\mathrm{an}}$}
\begin{tikzcd}
 K^{\an}(A) \ar[d]\ar[r] & K^{\an}(\ol A) \ar[d] \\ 
 K^{\an}(B) \ar[r] & K^{\an}(\ol B) 
\end{tikzcd}
\]
is weakly cartesian.
\end{thm}

\begin{proof}
Using nilinvariance of $K^{\an}$, we will reduce this to pro-excision for $K^{\cont}$ and $K^{\pi^{\infty}}$ of adic rings, using Theorem~\ref{thm.weakfibresequence}.

Choose noetherian, finite dimensional rings of definition as in Lemma~\ref{lem:milnor-square-ring-of-def},
and let $I_{0} = \ker(A_{0} \to \ol A_{0})$, $J_{0} = \ker(B_{0} \to \ol B_{0})$, so that $I_{0} \cong J_{0}$ via $\phi$.
For every $m\geq 1$ 
\[
(A_{0}/I_{0}^{m}) [\tfrac{1}{\pi}] \to (A_{0}/I_{0})[\tfrac{1}{\pi}] \cong \ol A
\]
is then a nilpotent extension of Tate rings and hence 
\[
K^{\an}((A_{0}/I_{0}^{m}) [\tfrac{1}{\pi}]) \to K^{\an}(\ol A)
\]
is a weak equivalence by \cite[Prop.~6.17]{KSTI}. From Theorem~\ref{thm.weakfibresequence} we therefore deduce a weak fibre sequence
\[
\prolim{m} K^{\pi^{\infty}}( A_{0}/I_{0}^{m} \on (\pi) ) \to \prolim{m} K^{\cont}(A_{0}/I_{0}^{m}) \to K^{\an}(\ol A)
\]
and similarly for $B_{0}$, $J_{0}$.
  So we get a weak fibre
  sequence of squares of pro-spectra
  \[
 \boxed{\pi^\infty}\to \boxed{\mathrm{cont}} \to \boxed{\mathrm{an}} .
\]
The theorem thus follows from the following two
pro-excision   Claims \ref{claim1}, \ref{claim2}.
\end{proof}

\begin{claim}
	\label{claim1}
The square of pro-spectra
\[
  \tag{$\boxed{\mathrm{cont}}$}
\begin{tikzcd}
K^{\cont}(A _{0}) \ar[r] \ar[d,"\phi"] & \prolim{m} K^{\cont}(A_{0}/I_{0}^{m}) \ar[d] \\
K^{\cont}(B_{0}) \ar[r] & \prolim{m} K^{\cont}(B_{0}/J_{0}^{m}).
\end{tikzcd}
\]
is weakly cartesian.
\end{claim}

\begin{claim}
	\label{claim2}
The square of pro-spectra
\[
    \tag{$\boxed{\pi^\infty}$}
\begin{tikzcd}
K^{\pi^{\infty}}(A _{0} \on (\pi)) \ar[r] \ar[d,"\phi"] & \prolim{m} K^{\pi^{\infty}}(A_{0}/I_{0}^{m} \on (\pi)) \ar[d] \\
K^{\pi^{\infty}}(B_{0} \on (\pi)) \ar[r] & \prolim{m} K^{\pi^{\infty}}(B_{0}/J_{0}^{m} \on (\pi)).
\end{tikzcd}
\]
is weakly cartesian.
\end{claim}

\begin{proof}[Proof of Claim~\ref{claim1}]
As $I_{0} \xrightarrow{\sim} J_{0}$ via $\phi$, also $I_{0}^{m} \xrightarrow{\sim} J_{0}^{m}$.
As $\pi$ is a non-zero divisor on $A_{0}/I_{0}$, it is also a non-zero divisor on $A_{0}/I_{0}^{m}$ for all $m$, and similarly on $B_{0}/J_{0}^{m}$. It follows from these two observations that for every pair of positive integers $n$, $m$ we have a Milnor square
\begin{equation}
	\label{mmm}
\begin{tikzcd}
A _{0}/(\pi^{n})\ar[r] \ar[d,"\phi"] & A_{0}/(\pi^{n}, I_{0}^{m}) \ar[d] \\
B_{0}/(\pi^{n}) \ar[r] & B_{0}/(\pi^{n},J_{0}^{m}).
\end{tikzcd}
\end{equation}
Pro-excision for noetherian rings \cite[Cor.~0.4]{Morrow} (see also \cite[Sec.~2.4]{LandTamme} for an account) implies that the square of $K$-theory pro-spectra
\[
\prolim{m} K(\eqref{mmm})
\]
is weakly contractible for every $n$. Taking the limit over $n$ gives the claim.
\end{proof}

\begin{proof}[Proof of Claim~\ref{claim2}]
From pro-excision for noetherian rings, applied to the Milnor squares
\begin{equation}
	\label{nnn}
\begin{tikzcd}
A _{0}[\Delta^{p}_{\pi^{j}}]\ar[r] \ar[d,"\phi"] & A_{0}/I_{0}^{m} [\Delta^{p}_{\pi^{j}}]\ar[d] \\
B_{0} [\Delta^{p}_{\pi^{j}}]\ar[r] & B_{0}/J_{0}^{m}[\Delta^{p}_{\pi^{j}}] ,
\end{tikzcd}
\end{equation}
we deduce that for fixed $p$ and $j$, the square of pro-spectra
\[
\prolim{m} K( \eqref{nnn} \on (\pi) ) 
\]
is weakly cartesian. As in the proof of Theorem~\ref{thm:fund-fib-model}, using finite dimensionality and Weibel's conjecture we deduce that we can pass to geometric realizations, and deduce the claim.
\end{proof}

\section{Descent and applications}
\label{sec:descent}

\subsection{Descent for analytic \textit{K}-theory}

Let $(A, A^{+})$ be an affinoid ring in the sense of \cite[Sec.~3]{Huber}, and let $X = \Spa(A,A^{+})$. We assume that $A$ is Tate and that it admits a noetherian, finite dimensional ring of definition $A_{0}$. If $U \subseteq X$ is a rational subdomain, then the same conditions hold for the affinoid ring $(\cO_{X}(U), \cO_{X}(U)^{+})$. We fix a pseudo-uniformizer $\pi$ of $A$. Then (the image of) $\pi$ is also a pseudo-unifomizer for every Tate ring $\cO_{X}(U)$.
For ease of notation we set
\[
K^{\an}(U) := K^{\an}(\cO_{X}(U)).
\]
Let $f_{1}, \dots, f_{n} \in A$ be elements generating $A$ as an ideal. These give rise to the standard rational covering $\cU = (U_i)_{1\leq i\leq n}$ of $X$ 
where 
\[
U_{i} = X(\frac{f_{1}, \dots, f_{n}}{f_{i}}).
\]
Write
\begin{equation}\label{eq;KanCechsieve}
 K^{\an}(|\cU_{\bullet}| ) := \lim_{p \in \Delta} \Kan(\cU_{p}),
\end{equation}
where $\cU_{\bullet}$ is the \v{C}ech nerve of the covering $\cU$.

\begin{thm}
	\label{thm:descent}
In the above situation, the natural map  
\[ 
\Kan(A) \to \Kan(|\cU_{\bullet}| ) 
\]
is a weak equivalence of pro-spectra.
\end{thm}

\begin{proof}
Fix a finite dimensional, noetherian ring of definition $A_{0}$ of $A$.
Multiplying by a power of $\pi$, we may assume that the $f_{i}$ belong to $A_{0}$. 
Let $I \subseteq A_{0}$ be the ideal generated by $f_{1}, \dots, f_{n}$, and let
\[
Y = \Proj(\bigoplus_{k\geq 0} I^{k})
\]
be the blowup of $\Spec(A_{0})$ along the ideal $I$. 
Then $Y \to \Spec(A_{0})$ is an admissible morphism.
Let $V_{i} = D_{+}(f_{i}) \subseteq Y$.
Then
$\cV = (V_{i})_{1\leq i\leq n}$ is an affine open covering of $Y$.  The $\pi$-adic completion $\cO_{Y}(V_{i})^{\wedge}$ can be canonically identified with a ring of definition of the Tate ring $\cO_{X}(U_{i})$, and similarly for the higher intersections $V_{i_{0}} \cap \dots \cap V_{i_{p}}$.

Consider the following diagram of pro-spectra
\begin{equation}\label{diag:descent}
\begin{split}
\xymatrix{
K^{\pi^{\infty}}(Y \on (\pi)) \ar[r]\ar[d] 
    		& \Kcont(Y) \ar[r]\ar[d]  & \Kan(A) \ar[d] \\
K^{\pi^{\infty}}(|\cV_{\bullet}| \on (\pi)) \ar[r] 
    		& \Kcont(|\cV_{\bullet}|) \ar[r] & \Kan(|\cU_{\bullet}|)\;,
}
\end{split}
\end{equation}
where $K^{\cont}(|\cV_{\bullet}|)$ and $K^{\pi^{\infty}}(|\cV_{\bullet}|\on (\pi))$ are defined as \eqref{eq;KanCechsieve}.
The top line is a weak fibre sequence by Theorem~\ref{thm:fund-fib-model}.
Note that the coverings $\cU$ and $\cV$ consist of $n$ open subsets of $X$ and $Y$, respectively, and thus their \v{C}ech nerves are $n$-skeletal. 
This implies that the limit
\[
K^{\cont}(|\cV_{\bullet}|) = \lim_{[p]\in\Delta} K^{\cont}(\cV_{p})
\]
in $\Pro(\Sp)$ can be computed level-wise (see e.g.~\cite[Lem.~2.1]{KSTI}) and similarly for $K^{\pi^{\infty}}(|\cV_{\bullet}| \on (\pi))$.
From Zariski descent for non-connective $K$-theory \cite[Thm.~8.1]{thomason} we thus deduce that the left two vertical maps in \eqref{diag:descent} are equivalences of pro-spectra.

Let 
\[
V_{i}' = \Spec(\cO_{Y}(V_{i})^{\wedge})
\]
and similarly for the higher intersections.
By Thomason's excision \cite[Prop.~3.19]{thomason} we have an equivalence
\[
K^{\pi^{\infty}}(|\cV_{\bullet}|\suppi) \simeq K^{\pi^{\infty}}(|\cV'_{\bullet}|\suppi),
\]
and we clearly also have $\Kcont(|\cV_{\bullet}|) \simeq \Kcont(|\cV'_{\bullet}|)$. Hence we deduce from Theorem~\ref{thm.weakfibresequence}
that the lower horizontal line in \eqref{diag:descent} is also a fibre sequence. We conclude that  the right vertical map is a weak equivalence, as was to be shown.
\end{proof}

\subsection{Globalization}

In order to globalize the construction of analytic $K$-theory, we first recall some sheaf theory.
Let $C$ be a small category. We denote by
\[
\PSh(C) = \Fun({C^{op}, \Spaces})
\]
 the $\infty$-category of presheaves of spaces on $C$ and by $y \colon C \to \PSh(C)$ the Yoneda embedding. If $F$ and $U$ are presheaves, we write 
 \[
 F(U) = \Map_{\PSh(C)}(U,F).
 \]
Note that by the Yoneda lemma, we have $F(X) = F(y(X))$ for every object $X$ of $C$.

If $\bD$ is any complete $\infty$-category, then composition with the Yoneda embedding induces an equivalence 
\begin{equation}\label{eq:presheaves-limit-preserving-functors}
\Fun^{lim}(\PSh(C)^{op}, \bD) \xrightarrow{\simeq} \Fun(C^{op},\bD)
\end{equation}
where $\Fun^{lim}(\PSh(C)^{op}, \bD) \subseteq \Fun(\PSh(C)^{op}, \bD)$ denotes the full subcategory spanned by the limit preserving functors
(see \cite[Thm.~5.1.5.6]{HTT}).

We now assume that $C$ is equipped with a (Grothendieck) topology. 
There is a bijection between sieves on an object $X \in C$ and subobjects $U \hookrightarrow y(X)$ of $y(X)$ in $\PSh(C)$ \cite[Prop.~6.2.2.5]{HTT}. For simplicity, we call the latter sieves, too. 
By definition, a presheaf $F \in \PSh(C)$ is a sheaf if and only if the natural map 
\[
F(X) \to F(U)
\]
is an equivalence  for every $X \in C$ and every covering sieve $U \hookrightarrow y(X)$ of $X$.
We denote by 
\[
\Sh(C) \subseteq \PSh(C)
\]
the full subcategory consisting of sheaves. 
The inclusion $\Sh(C) \hookrightarrow \PSh(C)$ admits a  left adjoint $L\colon \PSh(C) \to \Sh(C)$ called sheafification \cite[Prop.~5.5.4.15, Lemma~6.2.2.7]{HTT}. 
We say that a functor $F \in \Fun(C^{op}, \bD)$ is a ($\bD$-valued) sheaf if $F(X) \to F(U)$ is an equivalence in $\bD$ for every $X$ and covering sieve $U\hookrightarrow y(X)$ as above. Here $F(U)$ is defined using~\eqref{eq:presheaves-limit-preserving-functors} above.
By \cite[Prop.~5.5.4.20]{HTT} composition with $L$ induces a fully faithful functor
\[
\Fun^{lim}(\Sh(C)^{op}, \bD) \to \Fun^{lim}(\PSh(C)^{op}, \bD) \simeq \Fun(C^{op}, \bD)
\]
whose essential image consists precisely of the $\bD$-valued sheaves.

We now  return to analytic $K$-theory of Tate rings. Fix a complete affinoid Tate ring $(R, R^{+})$ such that $R$ admits a finite dimensional, noetherian ring of definition $R_{0}$, and let $S = \Spa(R,R^{+})$. We denote by 
\[
\Ad_{S}^{\aff, \ft} \subseteq  \Ad_{S}^{\sep, \lft} \subseteq \Ad_{S}^{\lft}
\]
the categories of affinoid adic spaces of finite type over $S$, of adic spaces separated and locally of finite type over $S$, and of adic spaces locally of finite type over $S$, respectively.
See \cite[Sec.~3]{Huber-gen} and \cite[Sec.~1.2, 1.3]{Huber-book} for these notions.
Note that for any $X \in \Ad_{S}^{\lft}$ and any open affinoid $U=\Spa(A,A^{+}) \subseteq X$, $A$ admits a finite dimensional, noetherian ring of definition. Indeed, by \cite[Prop.~3.6]{Huber-gen} the map $(R,R^{+}) \to (A, A^{+})$ is topologically of finite type. By Lemma 3.3(iii) there, $R \to A$ factors through a continuous open surjection $R\langle x_{1}, \dots x_{n} \rangle \to A$, and we can take the image of $R_{0}\langle x_{1}, \dots, x_{n}\rangle$ in $A$ as a ring of definition.

We equip the categories of adic spaces above with the usual topologies, see \cite[Sec.~2]{Huber}.
We then have:
\begin{lemma}
	\label{lem:sheaves}
The inclusion $\Ad_{S}^{\aff,\ft} \subseteq \Ad_{S}^{\lft}$ induces an equivalence 
\[
\Sh( \Ad_{S}^{\lft} ) \simeq \Sh( \Ad_{S}^{\aff,\ft}).
\]
\end{lemma}
\begin{proof}
This follows by applying \cite[Lemma~C.3]{Hoyois-trace-formula} to both inclusions $\Ad_{S}^{\aff, \ft} \subseteq  \Ad_{S}^{\sep, \lft} \subseteq \Ad_{S}^{\lft}$. Here we use that every $X \in \Ad_{S}^{\lft}$ has a covering $\{ U_{i} \to X\}$ such that all intersections $U_{i_{0}} \cap \dots \cap U_{i_{n}}$ are separated, and any separated $X$ has such a covering with all intersections affinoid.
\end{proof}

\begin{cor}
The presheaf
\[
K^{\an}\colon \Ad_{S}^{\aff,\ft,op} \to \Pro(\Sp^{+}), \quad \Spa(A,A^{+}) \mapsto K^{\an}(A),
\]
is a sheaf. In particular, it extends essentially uniquely to a sheaf with values in $\Pro(\Sp^{+})$ on the category $\Ad^{\lft}_{S}$.
\end{cor}

In particular, we can now make sense of the analytic $K$-theory $K^{\an}(X)$ for every adic space $X$ locally of finite type over the base $S$.

\begin{proof}
In view of the general discussion at the beginning of this subsection the second assertion follows from the first one and  Lemma~\ref{lem:sheaves}.
Now we prove the first one.
For each affinoid $X$ we define a family of sieves $\tau(X)$ on $X$ as follows: A sieve $U \hookrightarrow y(X)$ belongs to $\tau(X)$ if and only if the natural map 
\[
K^{\an}(X') \to K^{\an}(U \times_{y(X)} y(X'))
\]
is an equivalence in $\Pro(\Sp^{+})$ for every morphism $X' \to X$ in $\Ad^{\aff,\ft}_{S}$. Then the association $X \rightsquigarrow \tau(X)$ defines a topology on $\Ad^{\aff,\ft}_{S}$, see \cite[Prop.~C.1]{Hoyois-trace-formula}.
Since the topology on $\Ad^{\aff,\ft}_{S}$ is generated by rational coverings, it suffices to show that the sieve generated by a rational covering of $X$ belongs to $\tau(X)$. Since rational coverings are stable under pullback, it suffices to show that 
\[
K^{\an}(X) \to K^{\an}(|U_{\bullet}|)
\]
is an equivalence, where $|U_{\bullet}|  = \colim_{[p]\in\Delta^{op}} y(U_{[p]})$ is the realization of the \v{C}ech nerve of a rational covering  of $X$. Note here that $|U_{\bullet}| \hookrightarrow y(X)$ is precisely the sieve generated by that covering.
This equivalence is a consequence of Theorem~\ref{thm:descent}.
\end{proof}

\subsection{Properties of analytic \textit{K}-theory}

In this subsection we record some immediate consequences. We keep the assumptions on the affinoid adic base space $S$ as in the previous subsection.

\begin{prop} [{$\A^{1}$-invariance}] 
Let $X$ be an adic space locally of finite type over $S$, and let $\A^{1}_{X}$ denote the adic affine line over $X$. Then the natural map
\[
K^{\an}(X) \to K^{\an}(\A^{1}_{X})
\]
is an equivalence in $\Pro(\Sp^{+})$.
\end{prop}
\begin{proof}
It suffices to check this when $X = \Spa(A, A^{+})$ is affinoid. Let 
$X\langle t\rangle = \Spa(A\langle t\rangle, A^{+}\langle t\rangle)$.
The sheaf (of spaces) $y(\A^{1}_{X})$ is equivalent to  
$\colim_{t \mapsto \pi t} y(X\langle t\rangle)$.
Hence
\[
K^{\an}(\A^{1}_{X}) \simeq \lim_{t \mapsto \pi t} K^{\an}(X\langle t\rangle) \simeq K^{\an}(X)
\]
where the second equivalence follows from Corollary~\ref{cor:pro-homotopy-invariance}.
\end{proof}

Recall that the $K$-module structure of $\Kan$ as a functor on affinoids provides a natural transformation $\mycup t \colon \Kan \to \Omega\Kan(\underline{\phantom{x}}\langle t,t^{-1}\rangle)$. It induces a natural transformation on the extension of $\Kan$ to all adic spaces in $\Ad_{S}^{\lft}$.

\begin{prop} [Projective line]
Let $X$ be an adic space locally of finite type over $S$, and denote by $\bP^{1}_{X}$ the adic projective line over $X$. We have a natural weak equivalence
\[
\Kan(\bP^{1}_{X}) \simeq \Kan(X) \oplus \Kan(X).
\]
\end{prop}
\begin{proof}
We may again assume that $X$ is affinoid.
Consider the  covering of $\bP^{1}_{X}$ given by 
$X\langle t\rangle, X\langle t^{-1}\rangle $. 
Denote the projection $\bP^{1}_{X} \to X$ by $p$ and let $q:=p|_{X\langle t,t^{-1}\rangle}$.
Since $\Kan$ is a sheaf, the upper line in the following diagram is a fibre sequence.
\[
\xymatrix{
\Omega \Kan(X\langle t,t^{-1}\rangle) \ar[r]   & \Kan(\bP^{1}_{X}) \ar[r]    & \Kan(X\langle t\rangle) \oplus \Kan(X\langle t^{-1}\rangle) \\
\Omega\Gamma \Kan(X) \ar[u]^{q^{*}} \ar[r] & \Kan(X) \ar[u]^{p^{*}} \ar[r]    & \Kan(X\langle t\rangle) \oplus \Kan(X\langle t^{-1}\rangle) \ar[u]^{\simeq}
}
\]
Here
\[
\Gamma \Kan (X) = \Kan(X\langle t\rangle) \sqcup_{K^{\an}(X)} \Kan(X\langle t^{-1}\rangle)
\]
The lower line is a fibre sequence by definition of $\Gamma \Kan(X)$.  Since $p^{*}$ is split, we have an equivalence 
$\Kan(\bP^{1}_{X}) \simeq \Kan(X) \oplus \cofib(p^{*})$. 
From the diagram we get a weak equivalence $\cofib(q^{*}) \simeq \cofib(p^{*})$. By definition of $\Lambda K^{\an}$ (see Subsection~\ref{sec:fund-fib}) we have an equivalence $\cofib(q^{*}) \simeq \Lambda K^{\an}(A)$.
By Corollary~\ref{cor:BassFT} the composite 
\[
\Kan(X) \xrightarrow{\mycup t} \Omega \Kan(X\langle t,t^{-1}\rangle) \to \Lambda \Kan(X) \simeq \cofib(q^{*}) \simeq \cofib(p^{*})
\]
is a weak equivalence. Thus this map together with $p^{*}$ induces the desired weak equivalence $\Kan(X) \oplus \Kan(X) \xrightarrow{\simeq} \Kan(\bP^{1}_{X})$.
\end{proof}

We can globalize Theorem~\ref{thm:fund-fib-model} as follows.
Recall that $(R,R^{+})$ denotes an affinoid Tate ring such that $R$ admits a noetherian, finite dimensional ring of definition $R_{0}$, $\pi \in R_{0}$ is a pseudo-uniformizer, and $S=\Spa(R,R^{+})$.
Denote by $\Sch_{R_{0}}^{\ft}$ the category of  $R_{0}$-schemes of  finite type.
There is a functor 
\[
\Sch_{R_{0}}^{\ft} \to \Ad_{S}^{\lft}, \quad X \mapsto X^{\ad},
\]
which on affine  schemes is given by $\Spec(A_{0}) \mapsto \Spa(A,A^{\circ})$ with $A=A_{0}^{\wedge}[\frac{1}{\pi}]$ and $A^{\circ} \subseteq A$ the subring of powerbounded elements (compose the $\pi$-adic completion functor, the functor $t$ from \cite[Prop.~4.1]{Huber-gen}, which sends $\Spf(A_{0}^{\wedge})$ to $\Spa(A_{0}^{\wedge},A_{0}^{\wedge})$, and the functor $Y \mapsto Y(\pi\not=0)$).

\begin{thm}
For $X$ a scheme of finite type over $R_{0}$, there is a weak fibre sequence of pro-spectra
\[
K^{\pi^{\infty}}(X \on (\pi)) \to K^{\cont}(X) \to K^{\an}(X^{\ad}).
\]
If there exists an admissible morphism $\widetilde X \to X$ with $\widetilde X$ regular, then there is  a weak fibre sequence
\[
K(X \on (\pi)) \to K^{\cont}(X) \to K^{\an}(X^{\ad}).
\]
\end{thm}
\begin{proof}
We equip $\Sch_{R_{0}}^{\ft}$ with the Zariski topology. 
Considered as functors on $\Sch_{R_0}^{\ft,op}$, all pro-spectra appearing in the statement of the Theorem are sheaves (see the proof of Theorem~\ref{thm:descent} for the cases of $K^{\pi^{\infty}}(\underline{\phantom{x}}\suppi)$ and $\Kcont$). Since every (regular) $X \in \Sch^{\ft}_{R_0}$ 
is a colimit of (regular) affine schemes in the category of sheaves $\Sh(\mathbf{Sch}^{\ft}_{R_0})$, it suffices to check the first statement in the affine case. But this is Theorem~\ref{thm.weakfibresequence}. 

If there exists an admissible morphism $\widetilde X \to X$ with $\widetilde X$ regular, then \cite[Cor.~4.7]{KSTI} implies that 
for every $p$, the map 
\[
K(X \on (\pi)) \to  \prolim{j} K(X\otimes_{R_{0}} R_{0}[\Delta^{p}_{\pi^{j}}] \on (\pi))
\]
is a weak equivalence of pro-spectra.
Using arguments similar to those in the proof of Theorem~\ref{thm.weakfibresequence},  Weibel's $K$-dimension conjecture implies that $K^{\pi^{\infty}}(X \on (\pi))$ is weakly equivalent to $K(X \on (\pi))$. This establishes the second claim.
\end{proof}

\bibliographystyle{amsalpha}
\bibliography{Kanalytic}

\end{document}